\documentclass[10pt]{article}
\usepackage[centertags]{amsmath}
\usepackage{amsmath, amsthm,amssymb,mathtools}
\usepackage{fullpage}
\usepackage{biblatex}
\usepackage{tikz}
\usepackage{physics}
\usepackage{amsmath}
\usepackage{mathdots}
\usepackage{yhmath}
\usepackage{cancel}
\usepackage{color}
\usepackage{siunitx}
\usepackage{array}
\usepackage{multirow}
\usepackage{amssymb}
\usepackage{gensymb}
\usepackage{tabularx}
\usepackage{extarrows}
\usepackage{booktabs}
\usepackage[shortlabels]{enumitem}
\usepackage{float}
\usetikzlibrary{fadings}
\usetikzlibrary{patterns}
\usetikzlibrary{shadows.blur}
\usetikzlibrary{shapes}
\usepackage[font=footnotesize,labelfont=bf, justification=centering, margin=2cm]{caption}
\usepackage[margin=1in, footskip=0.25in]{geometry}
\usepackage{setspace}

\newcommand{\R}{\mathbb{R}}

\newcommand{\Z}{\mathbb{Z}}

\usepackage{fancyhdr}
\makeatletter
\renewcommand*\env@matrix[1][*\c@MaxMatrixCols c]{%
  \hskip -\arraycolsep
  \let\@ifnextchar\new@ifnextchar
  \array{#1}}
\makeatother
\theoremstyle{definition}
\newtheorem{defn}{Definition}[section]

\theoremstyle{plain}

\newtheorem{thm}[defn]{Theorem}

\newtheorem{prop}[defn]{Proposition}
\newtheorem{lem}[defn]{Lemma}
\theoremstyle{remark}

\title{LLL Algorithm for Lattice Basis Reduction}
\author{Alex Kalbach, Ted Chinburg}
\date{\vspace{-5ex}}

\begin{document}
\maketitle
\vspace{8pt}
{\footnotesize This is an expository paper intended to introduce the polynomial time lattice basis reduction algorithm first described by Arjen Lenstra, Hendrik Lenstra, and László Lovász in 1982. We begin by introducing the shortest vector problem, which motivates the underlying components of the LLL algorithm. Then, we introduce the details of the algorithm itself, followed by proofs of the correctness and runtime of the algorithm in complete detail, assuming only a basic linear algebra background and an understanding of big O notation. Finally, we apply the LLL algorithm to the shortest vector problem and explore other applications of the algorithm in various mathematical settings.}

\section{Shortest vector problem}
\begin{defn}
\sloppy For any ordered set of linearly independent vectors $B = \{\vec{b}_1, \vec{b}_2, \ldots, \vec{b}_n\}$ in $\R^k$, we can define the \textit{lattice} of $B$ as the set
\[
\Lambda(B) = \left\{v \, : \, v = c_1\vec{b}_1 + c_2\vec{b}_2 + \cdots + c_n\vec{b}_n,\, c_i \in \Z \right\}.
\]
\end{defn}
We say that the set $B$ is the \textit{basis} of the rank $n$ lattice $\Lambda(B)$. For example, the lattice with basis
\[
\left\{\begin{pmatrix}0\\ 1
\end{pmatrix}, \begin{pmatrix}1\\0
\end{pmatrix}
\right\}
\]

is the set of all vectors in $\Z^2$. Note that the number of vectors in the basis, $n$, is not necessarily equal to the dimension of each vector, $k$, but it must be that $n \leq k$ in order for $B$ to be a linearly independent set.

The shortest vector problem (SVP) asks one to find, given any lattice $\Lambda(B)$, the shortest non-zero vector contained in $\Lambda(B)$. For simple lattices like the one above, it is clear that the shortest vector is one of the given basis vectors, with length 1. But, for more complicated lattices, like one with basis
\[
\left\{\begin{pmatrix}-2\\3\\1\\2
\end{pmatrix}, \begin{pmatrix}3\\-1\\1\\-2
\end{pmatrix}, \begin{pmatrix}-6\\-5\\-6\\2
\end{pmatrix}
\right\},
\]
it is not as obvious that a shortest vector is
\[
3\vec{b}_1 + 4\vec{b}_2 + \vec{b}_3 = \begin{pmatrix}0\\0\\1\\0
\end{pmatrix},
\] also of length $1$.

It is believed to be an NP-hard problem to find a vector in a lattice with length within a factor of $\sqrt{2}$ of the length of the shortest vector, let alone finding the shortest vector itself [1]. This would mean that if there was a polynomial time algorithm that could solve the SVP, even to just a $\sqrt{2}$-factor approximation, we could then solve any NP problem (a problem which has a proof verifiable in polynomial time) in polynomial time.

It is for reasons such as this that lattices are used as the core of certain encryption algorithms. These algorithms rely on both the ease of constructing a lattice of vectors used for encryption, and the difficulty of obtaining the shortest vectors of a given lattice, which would aid in decryption. Applications of lattices in cryptography and other areas of mathematics will be briefly mentioned in section 6.

\section{Motivation for the LLL algorithm}

The Lenstra-Lenstra-Lovàsz (LLL) algorithm aims to iteratively reduce vectors in a basis for a lattice in order to find short vectors in that lattice. As later shown in Theorem 5.2, for an $n$-dimensional lattice, the reduced basis will contain a vector that is within an exponential factor in $n$ of being the shortest vector. The exposition of the LLL algorithm in this paper is largely based on the description given in the notes of Micciancio in [2], which are in turn based on the original paper by Lenstra, Lenstra, and Lovasz [3]. We have also benefitted from the descriptions of the LLL algorithm given by Deng [4] and Etienne [5].

The core idea behind the algorithm is to take a basis of a lattice and repeatedly change the vectors in that basis to make them consistently smaller, while still remaining within the lattice. An important step of this is reducing the projection coefficients between pairs of vectors. The intuition behind why we want to do this is as follows:

\begin{prop}
Given two vectors $\vec{u}, \vec{v} \in  \Lambda(B)$, we can define the projection coefficient
\[\mu = \frac{\langle\vec{u}, \vec{v}\rangle}{\langle\vec{v}, \vec{v}\rangle}.\] 
Then, we have that
\[\left\Vert\vec{u} - \lfloor\mu\rceil\vec{v}\right\Vert \leq \Vert\vec{u}\Vert,\]
where $\lfloor x \rceil$ denotes the nearest integer to $x$. If there are two integers equally close, the one with lesser absolute value is selected (e.g. 1/2 and -1/2 round to 0).
\end{prop}

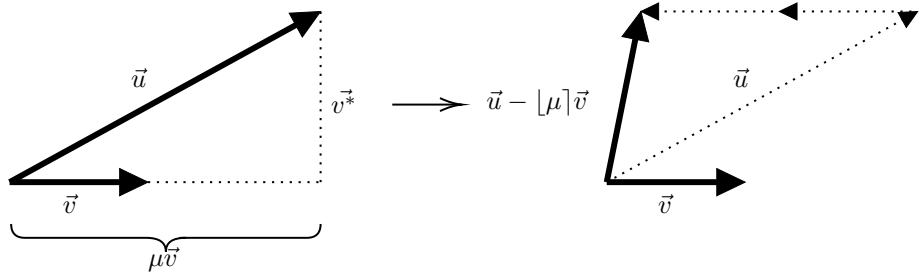
\begin{figure}[H]
\centering
\tikzset{every picture/.style={line width=0.75pt}} 

\begin{tikzpicture}[x=0.75pt,y=0.75pt,yscale=-1,xscale=1]

\draw [line width=2.25]    (101,235.15) -- (165.8,235.15) ;
\draw [shift={(170.8,235.15)}, rotate = 180] [fill={rgb, 255:red, 0; green, 0; blue, 0 }  ][line width=0.08]  [draw opacity=0] (14.29,-6.86) -- (0,0) -- (14.29,6.86) -- cycle    ;
\draw [line width=2.25]    (101,235.15) -- (253.47,151.41) ;
\draw [shift={(257.85,149)}, rotate = 151.22] [fill={rgb, 255:red, 0; green, 0; blue, 0 }  ][line width=0.08]  [draw opacity=0] (14.29,-6.86) -- (0,0) -- (14.29,6.86) -- cycle    ;
\draw  [dash pattern={on 0.84pt off 2.51pt}]  (170.8,235.15) -- (257.85,235.15) ;
\draw  [dash pattern={on 0.84pt off 2.51pt}]  (257.85,235.15) -- (257.85,149) ;
\draw   (101.78,256.25) .. controls (101.78,260.92) and (104.11,263.25) .. (108.78,263.25) -- (169.82,263.25) .. controls (176.49,263.25) and (179.82,265.58) .. (179.82,270.25) .. controls (179.82,265.58) and (183.15,263.25) .. (189.82,263.25)(186.82,263.25) -- (250.85,263.25) .. controls (255.52,263.25) and (257.85,260.92) .. (257.85,256.25) ;
\draw [line width=2.25]    (402.15,235.15) -- (466.95,235.15) ;
\draw [shift={(471.95,235.15)}, rotate = 180] [fill={rgb, 255:red, 0; green, 0; blue, 0 }  ][line width=0.08]  [draw opacity=0] (14.29,-6.86) -- (0,0) -- (14.29,6.86) -- cycle    ;
\draw [line width=0.75]  [dash pattern={on 0.84pt off 2.51pt}]  (402.15,235.15) -- (556.37,150.44) ;
\draw [shift={(559,149)}, rotate = 151.22] [fill={rgb, 255:red, 0; green, 0; blue, 0 }  ][line width=0.08]  [draw opacity=0] (8.93,-4.29) -- (0,0) -- (8.93,4.29) -- cycle    ;
\draw [line width=0.75]  [dash pattern={on 0.84pt off 2.51pt}]  (492.2,149) -- (559,149) ;
\draw [shift={(489.2,149)}, rotate = 0] [fill={rgb, 255:red, 0; green, 0; blue, 0 }  ][line width=0.08]  [draw opacity=0] (8.93,-4.29) -- (0,0) -- (8.93,4.29) -- cycle    ;
\draw [line width=0.75]  [dash pattern={on 0.84pt off 2.51pt}]  (422.4,149) -- (489.2,149) ;
\draw [shift={(419.4,149)}, rotate = 0] [fill={rgb, 255:red, 0; green, 0; blue, 0 }  ][line width=0.08]  [draw opacity=0] (8.93,-4.29) -- (0,0) -- (8.93,4.29) -- cycle    ;
\draw [line width=2.25]    (402.15,235.15) -- (418.42,153.9) ;
\draw [shift={(419.4,149)}, rotate = 101.32] [fill={rgb, 255:red, 0; green, 0; blue, 0 }  ][line width=0.08]  [draw opacity=0] (14.29,-6.86) -- (0,0) -- (14.29,6.86) -- cycle    ;
\draw [line width=0.75]    (294,196) -- (327,196) ;
\draw [shift={(329,196)}, rotate = 180] [color={rgb, 255:red, 0; green, 0; blue, 0 }  ][line width=0.75]    (10.93,-3.29) .. controls (6.95,-1.4) and (3.31,-0.3) .. (0,0) .. controls (3.31,0.3) and (6.95,1.4) .. (10.93,3.29)   ;

\draw (160.8,176.08) node [anchor=north west][inner sep=0.75pt]    {$\vec{u}$};
\draw (126.37,239.73) node [anchor=north west][inner sep=0.75pt]    {$\vec{v}$};
\draw (169.32,266.66) node [anchor=north west][inner sep=0.75pt]    {$\mu \vec{v}$};
\draw (261.72,186.61) node [anchor=north west][inner sep=0.75pt]    {$\vec{v^*}$};
\draw (464.38,176.08) node [anchor=north west][inner sep=0.75pt]    {$\vec{u}$};
\draw (426.52,239.73) node [anchor=north west][inner sep=0.75pt]    {$\vec{v}$};
\draw (340.08,186.61) node [anchor=north west][inner sep=0.75pt]    {$\vec{u} -\lfloor \mu \rceil \vec{v}$};
\end{tikzpicture}

\caption{The resultant vector $\vec{u} - \lfloor \mu \rceil \vec{v}$ is shorter than $\vec{u}$ (as guaranteed by Proposition 2.1) and is still within the lattice generated by $\vec{u}$ and $\vec{v}$.} \label{fig:M1}
\end{figure}

\begin{proof}
When $\left|\mu\right| \leq 1/2$, $\mu$ rounds to 0 and equality holds. When $\left|\mu\right| > 1/2$, the coefficient of $\vec{v}$ is nonzero. We can always rewrite $\vec{u} = \mu\vec{v} + \vec{v^*}$, where $\vec{v^*}$ is a vector orthogonal to $\vec{v}$ (see Figure 1 for an example). Then,
\[\left\Vert\vec{u} - \left\lfloor\mu\right\rceil\vec{v}\right\Vert = \Vert\left(\mu - \lfloor \mu\rceil\right)\vec{v} + \vec{v^*}\Vert\]
\[\leq \left\Vert \frac{1}{2}\vec{v} + \vec{v^*}\right\Vert < \Vert \mu\vec{v} + \vec{v^*}\Vert = \Vert\vec{u}\Vert.
\]
\end{proof}

This tells us that, whenever the absolute value of the projection coefficient is greater than 1/2, by subtracting off a certain multiple of the vector $\vec{v}$ from $\vec{u}$, we obtain a new vector within the lattice that has a smaller projection coefficient (with absolute value $\leq 1/2$) and is shorter than the original $\vec{u}$. So, intuitively, to make a basis with smaller vectors, we should aim to reduce the projection coefficients between vectors, thus making the basis as close to orthogonal as possible.

This operation is essentially ``pseudo-orthogonalization'' within a lattice, and it is used as a core part of the LLL algorithm to reduce a basis. It is similar in process to Gram-Schmidt orthogonalization, which produces vectors that are exactly orthogonal to one another, though at the cost of those vectors not being in the original lattice.

\begin{defn} The Gram-Schmidt orthogonalization $B^* = \{\vec{b}_1^*, \vec{b}_2^*, \ldots, \vec{b}_n^*\}$ of a basis $B = \{\vec{b}_1, \vec{b}_2, \ldots, \vec{b}_n\}$ is computed as follows:
\begin{align*}
    \vec{b}_1^* &= \vec{b}_1\\
    \vec{b}_2^* &= \vec{b}_2 - \frac{\langle\vec{b}_2, \vec{b}_1^*\rangle}{\langle\vec{b}_1^*, \vec{b}_1^*\rangle}\vec{b}_1^*\\
    \vec{b}_3^* &= \vec{b}_3 - \frac{\langle\vec{b}_3, \vec{b}_1^*\rangle}{\langle\vec{b}_1^*, \vec{b}_1^*\rangle}\vec{b}_1^*  - \frac{\langle\vec{b}_3, \vec{b}_2^*\rangle}{\langle\vec{b}_2^*, \vec{b}_2^*\rangle}\vec{b}_2^*\\
    &\,\,\,\,\vdots\\
    \vec{b}_n^* &= \vec{b}_n - \sum_{i=1}^{n-1}
    \frac{\langle\vec{b}_n, \vec{b}_i^*\rangle}{\langle\vec{b}_i^*, \vec{b}_i^*\rangle}\vec{b}_i^*.
\end{align*}
\end{defn}
Note that we do not normalize the orthogonalized vectors for our purposes, since their lengths preserve information about the lengths of the vectors before they were orthogonalized, which will be useful in the algorithm.

We now introduce the concept of a $\delta$-LLL reduced basis, which is the output of the LLL algorithm given any input basis.

\begin{defn} For any basis $B = \{\vec{b}_1, \vec{b}_2, \ldots, \vec{b}_n\}$, we can compute the Gram-Schmidt orthogonalized basis $B^* = \{\vec{b}_1^*, \vec{b}_2^*, \ldots, \vec{b}_n^*\}$ and define the projection coefficients
\[
    \mu_{i, j} = \frac{\langle\vec{b}_i, \vec{b}_j^*\rangle}{\langle\vec{b}_j^*, \vec{b}_j^*\rangle}.
\]
Then, $B$ is $\delta$\textit{-LLL reduced} for some $\delta \in (1/4, 1)$ if and only if it satisfies the following two conditions:
\begin{enumerate}[(1)]
    \item $|\mu_{i, j}| \leq 1/2$ whenever $i > j$. (projection condition)
    \item $(\delta-\mu_{i+1, i}^2)\Vert\vec{b}_i^*\Vert^2 \leq \Vert\vec{b}_{i+1}^*\Vert^2$ for all $i < n$. (ordering condition)
\end{enumerate}
\end{defn}

Why is such a basis desirable? Firstly, the projection condition forces vectors of higher index in the reduced basis to have small projection coefficients when projected onto Gram-Schmidt orthogonalized vectors of lower index. This means each vector is ``close to orthogonal'' to the earlier Gram-Schmidt vectors, which we hope are themselves close to the original vectors of the same index (since we want to produce a basis that is itself close to orthogonal). So, by forcing the vectors to be close to orthogonal to the Gram-Schmidt orthogonalized basis vectors, we force the vectors to attain close to the same shortness that the Gram-Schmidt orthogonalized vectors attain.

By the ordering condition, $\delta-\mu_{i+1, i}^2$ times the squared length of one orthogonalized vector can never be more than the squared length of the next orthogonalized vector, so the earlier vectors in the orthogonalized basis cannot be much longer than the later vectors. Thus, the condition forces the vectors in the Gram-Schmidt orthogonalized basis (and by consequence, the $\delta$-LLL reduced basis) to be more or less ordered by length. However, this ordering is not perfect, because the constant $\delta-\mu_{i+1, i}^2$ will be less than 1.

While it is not immediately obvious, this ordering condition is important, because there are undesirable situations in which a basis satisfying the projection condition can still have very long vectors. Consider that the projection coefficient of a short vector onto a long vector is usually quite low. Since we compare only the projection coefficients of higher-indexed vectors onto lower-indexed Gram-Schmidt orthogonalized vectors, we need the shorter vectors to have lower index in the basis so that the longer vectors are being projected onto the shorter vectors, forcing the long vectors to become shorter.

Since the order of the vectors in the basis is critical to the state of being $\delta$-LLL reduced, it is important that we always consider a basis to be an \textit{ordered} set or a list of vectors, such that the order is preserved between the end of one iteration of the loop in the algorithm to the beginning of the next.

The bounds (1/4, 1) for $\delta$ exist so that the LLL algorithm is \textit{guaranteed} to terminate in polynomial time for any lattice. The reasoning behind these particular bounds will become evident in section 4.

To illustrate the shortness of the vectors in a $\delta$-LLL reduced basis, take the following example: the basis
 \[
B = \left\{\begin{pmatrix}2\\3\\1
\end{pmatrix}, \begin{pmatrix}1\\2\\-1
\end{pmatrix}, \begin{pmatrix}-2\\-2\\2
\end{pmatrix}\right\}
\]
has a corresponding $\frac{3}{4}$-LLL reduced basis of
\[
\left\{\begin{pmatrix}-1\\0\\1
\end{pmatrix}, \begin{pmatrix}0\\2\\0
\end{pmatrix}, \begin{pmatrix}1\\1\\2
\end{pmatrix}\right\},
\]
which, when written in terms of the original basis vectors, is equal to
\[
\left\{\vec{b}_2 + \vec{b}_3, 2\vec{b}_2 + \vec{b}_3, \vec{b}_1 - \vec{b}_2\right\}.
\]
So, the reduced basis is made up of vectors within the lattice of the original basis, yet it clearly has much shorter vectors. It can be verified that this basis satisfies the above conditions.

\section{LLL algorithm}

The algorithm itself does exactly what one would expect: it first checks whether or not the current basis satisfies the projection condition. If it doesn't, the algorithm applies the ``pseudo-orthogonalization'' from Proposition 2.1 so that the condition is satisfied up to a certain index. Then, the algorithm checks if the ordering condition is satisfied. If it isn't, two vectors are swapped. Then, it repeats the whole process until we have a $\delta$-LLL reduced basis.

More explicitly, the LLL algorithm proceeds as follows, using the lattice basis that we want to reduce as the input:\medskip

\hrule\smallskip

Set the indexing variable $i = 2$ and repeat the following until $i > n$:
\begin{enumerate}[(1), topsep=0pt]
    \item Compute the Gram-Schmidt orthogonalization of the current basis.
    \item (Reduction step) For all $j$ from $i-1$ to 1, adjust the $i$th vector in the current basis so that the absolute value of the projection coefficient from $\vec{b}_i$ onto $\vec{b}_j^*$ is $\leq$ 1/2. That is, set $\vec{b}_i \leftarrow \vec{b}_i - \lfloor\mu_{i, j}\rceil\vec{b}_j$. For each value of $j$, we recalculate the projection coefficient $\mu_{i, j}$ using the updated $\vec{b}_i$. We do \textbf{not} need to recalculate the Gram-Schmidt orthogonalized basis after this step (see Lemma 3.1).
    \item (Swapping step) If the ordering condition $(\delta-\mu_{i, i-1}^2)\Vert\vec{b}_{i-1}^*\Vert^2 \leq \Vert\vec{b}_{i}^*\Vert^2$ is not satisfied between the $i$th vector and the $(i-1)$th vector in the Gram-Schmidt orthogonalized basis, swap $\vec{b}_{i-1}$ and $\vec{b}_i$ in the current basis and decrease $i$ by 1 (to a minimum of 2). If the condition is satisfied, increase $i$ by 1.
 \end{enumerate}
Output the current basis.

\medskip
\hrule\smallskip
 
 As an instructive example, we will now go through each step of the algorithm, applied to the basis we used earlier, to construct a $\frac{3}{4}$-LLL reduced basis. We start with
 \[
\left\{\begin{pmatrix}2\\3\\1
\end{pmatrix}, \begin{pmatrix}1\\2\\-1
\end{pmatrix}, \begin{pmatrix}-2\\-2\\2
\end{pmatrix}\right\}.
\]
First, $i = 2.$ The projection condition is satisfied, but the ordering condition is not satisfied, so $\vec{b}_1$ is swapped with $\vec{b}_2$. $i$ cannot be decreased further, so it stays at 2.
\[
\left\{\begin{pmatrix}1\\2\\-1
\end{pmatrix}, \begin{pmatrix}2\\3\\1
\end{pmatrix}, \begin{pmatrix}-2\\-2\\2
\end{pmatrix}\right\}
\]
$i = 2$. $\vec{b}_2$ is changed to $\vec{b}_2 - \vec{b}_1$ to satisfy the projection condition. The ordering condition is then satisfied, so $i$ is incremented to 3.
\[
\left\{\begin{pmatrix}1\\2\\-1
\end{pmatrix}, \begin{pmatrix}1\\1\\2
\end{pmatrix}, \begin{pmatrix}-2\\-2\\2
\end{pmatrix}\right\}
\]
$i = 3$. $\vec{b}_3$ is changed to $\vec{b}_3 + \vec{b}_1$ to satisfy the projection condition. The ordering condition is not satisfied, so $\vec{b}_2$ is swapped with $\vec{b}_3$ and $i$ is decreased to 2.
\[
\left\{\begin{pmatrix}1\\2\\-1
\end{pmatrix}, \begin{pmatrix}-1\\0\\1
\end{pmatrix}, \begin{pmatrix}1\\1\\2
\end{pmatrix}\right\}
\]
$i = 2$. The projection condition is satisfied, but the ordering condition is not satisfied, so $\vec{b}_1$ is swapped with $\vec{b}_2$. $i$ cannot be decreased further, so it stays at 2.
\[
\left\{\begin{pmatrix}-1\\0\\1
\end{pmatrix}, \begin{pmatrix}1\\2\\-1
\end{pmatrix}, \begin{pmatrix}1\\1\\2
\end{pmatrix}\right\}
\]
$i = 2.$ $\vec{b}_2$ is changed to $\vec{b}_2 + \vec{b}_1$ to satisfy the projection condition. The ordering condition is satisfied, so $i$ is increased to 3.
\[
\left\{\begin{pmatrix}-1\\0\\1
\end{pmatrix}, \begin{pmatrix}0\\2\\0
\end{pmatrix}, \begin{pmatrix}1\\1\\2
\end{pmatrix}\right\}
\]
$i = 3.$ Both conditions are satisfied, so $i$ is increased to 4 and the algorithm terminates, outputting the resulting basis:
\[
\left\{\begin{pmatrix}-1\\0\\1
\end{pmatrix}, \begin{pmatrix}0\\2\\0
\end{pmatrix}, \begin{pmatrix}1\\1\\2
\end{pmatrix}\right\}.
\]

 Observe that the LLL algorithm only adjusts the input basis using integer multiples of vectors within the basis --- that is, all operations on $B$ do not introduce any new lattice points. So, the lattice of the output $\delta$-LLL reduced basis is entirely contained in the lattice of the input basis. Note that the Gram-Schmidt orthogonalized basis $B^*$ does \textit{not} generally generate the same lattice as $B$ since we do not care about rounding projection coefficients, but we simply use the orthogonalized basis in the algorithm as a means for testing whether the current basis is reduced.

It is now important to address one potential issue: why do we not require a recomputation of the Gram-Schmidt orthogonalized basis after the reduction step of the algorithm, since we use the orthogonalized basis to verify the ordering condition in the swapping step? This is because the operations in the reduction step do not affect the Gram-Schmidt orthogonalized basis $B^*$. This fact will be useful later, so we will prove it in the following lemma:

\begin{lem}
Let $B^*$ be the Gram-Schmidt orthogonalized basis at the start of an iteration of the algorithm, and let $B'$ be the current (non-orthogonalized) basis after applying the reduction step. Then, the Gram-Schmidt orthogonalization of $B'$ is identical to $B^*$.
\end{lem}
\begin{proof}
 During any given iteration, $i$ is fixed. None of the vectors $\vec{b}_j$ for $j \not= i$ are modified in the reduction step, so when $j < i,$ the $\vec{b}_j^*$ will remain the same. When $j > i$, assuming $\vec{b}_i^*$ is unchanged, since each $\vec{b}_j^*$ is computed using orthogonalized vectors of lower index, which would all be unchanged, $\vec{b}_j^*$ will be unchanged. Therefore, it suffices to only show that $\vec{b}_i^*$ is unchanged.

We can first look at the projection of $\vec{b}_i$ onto the Gram-Schmidt orthogonalized vector $\vec{b}_{i-1}^*$. During the Gram-Schmidt orthogonalization process at the start of the iteration, we subtract $\mu_{i,i-1}\vec{b}_{i-1}^*$ from $\vec{b}_i$ in order to compute $\vec{b}_i^*$. All other subtractions are orthogonal to this and thus do not affect the projection onto $\vec{b}_{i-1}^*.$

If we instead were to apply the reduction step to produce $B'$, we would have subtracted an integer linear combination of the $\vec{b}_j$ from $\vec{b}_i$, for $j < i$. From Definition 2.2, adding the sum to both sides of the equation, we can see that any vector in the basis can be represented as a linear combination of Gram-Schmidt orthogonalized vectors. Specifically, since each basis vector with index less than $i-1$ can be represented as a sum of Gram-Schmidt orthogonalized vectors of index less than $i-1$ (which are orthogonal to $\vec{b}_{i-1}^*$), the only part being subtracted from $\vec{b}_i$ in the reduction step that affects the projection coefficient $\mu_{i, i-1}$ is $\lfloor\mu_{i, i-1}\rceil\vec{b}_{i-1}$. Subtracting this would leave a new projection coefficient of
\[
\mu'_{i,i-1} = \frac{\langle\vec{b}_i-\lfloor\mu_{i, i-1}\rceil\vec{b}_{i-1}, \vec{b}_{i-1}^*\rangle}{\langle\vec{b}_{i-1}^*, \vec{b}_{i-1}^*\rangle}
.\]
Using linearity of the inner product, and the fact that the projection of a vector onto its Gram-Schmidt orthogonalized counterpart must be 1, we get that
\[
\mu'_{i,i-1} = \mu_{i,i-1} - \lfloor\mu_{i, i-1}\rceil.
\]
Then, when applying Gram-Schmidt orthogonalization to $B'$, we subtract \textit{another} vector off of $\vec{b}_i$ to compute $\vec{b}_i^*$: in particular, we would need to subtract $\mu'_{i,i-1}\vec{b}_{i-1}^*$. Combining the two subtractions, this means that the component of $\vec{b}_i$ parallel to $\vec{b}_{i-1}^*$ being subtracted off to produce $\vec{b}_i^*$ would be 
\[(\left\lfloor\mu_{i,i-1}\right\rceil + \mu'_{i,i-1})\vec{b}_{i-1}^* = \mu_{i,i-1}\vec{b}_{i-1}^*,\]
which is the same as the result from the orthogonalization process on $B$. Therefore, $\vec{b}_i$ is reduced to having the same component parallel to $\vec{b}_{i-1}^*$ in both cases. Then, the same argument holds when you are considering the component of $\vec{b}_i$ in the direction of $\vec{b}^*_{i-2}$. This is because $\lfloor\mu_{i, i-1}\rceil$ becomes 0 after the reduction with respect to $\vec{b}^*_{i-1}$, so the only term that matters in what is being subtracted is $\lfloor\mu_{i, i-2}\rceil\vec{b}_{i-2}$. The argument can likewise be applied iteratively to the projection onto all Gram-Schmidt orthogonalized vectors with index less than $i$. 
\end{proof}

Now, before looking at the runtime of the algorithm, we must first verify that the LLL algorithm, if it has an output, actually gives us the reduced basis that we want.

\begin{prop} Any basis $B$ that is an output of the LLL algorithm is a $\delta$-LLL reduced basis.
\end{prop}

\begin{proof}If the algorithm terminates with an output, then $i$ would have attained each value $2, 3, \ldots, n$ for the last time at some iteration. We want to show that after $i$ attained the value $j$ for the last time, the vectors $\vec{b}_1, \vec{b}_2, \ldots, \vec{b}_j$ always satisfied the $\delta$-LLL reduced basis conditions. In the case when $i = 2$ for the last time, the first two vectors must have satisfied the ordering condition because otherwise they would have been swapped. Then, they must have satisfied the projection condition because when $i = 2$, the coefficient $\mu_{2, 1}$ was reduced to absolute value $\leq 1/2$ by the reduction step.

Now, assume that the vectors $\vec{b}_1, \vec{b}_2, \ldots, \vec{b}_j$ satisfy the conditions and $i$ never attains any value $\leq j+1$ later in the algorithm. Then, $\vec{b}_{j+1}$ must satisfy the ordering condition with $\vec{b}_j$ because otherwise they would be swapped when $i = j+1$, and then $i$ would decrease, which is a contradiction. All other vectors with lower index must satisfy the ordering condition as well because of our assumption. Also, $\vec{b}_{j+1}$ must satisfy the projection condition because at the last iteration in the reduction step, all $\mu_{j+1, \ell}$ were reduced to absolute value $\leq 1/2$ for $\ell < j+1$.

So, the claim holds for all $i = 2, 3, \ldots, n$ by induction. In particular, after $i = n$ for the last time in the algorithm (i.e. when the algorithm terminates), all vectors in the basis satisfy the $\delta$-LLL reduced basis conditions. Therefore, the algorithm produces a $\delta$-LLL reduced basis, as we hoped.
\end{proof}

We have shown that, given any input lattice, the LLL algorithm --- if it has an output --- will output a $\delta$-LLL reduced basis, which is made up of short (though not necessarily the shortest) vectors in the lattice.
Naturally, we want to also prove that the algorithm is \textit{guaranteed} to terminate for $\delta\in(1/4, 1)$, and furthermore, that the algorithm is guaranteed to terminate in \textit{polynomial time}.

\section{Polynomial time bound for LLL algorithm}

We now aim to show that the LLL algorithm will terminate in polynomial time as a function of the dimension, $k$, of the real space in which the input basis vectors lie. For this section, ``polynomially bounded'' or ``polynomial time'' refers to a polynomial in terms of the dimension $k$; that is, anything that runs in $O(k^m)$ bit operations for some constant $m$. For most of the following proofs, it suffices to show the algorithm is polynomial in the rank of the basis, $n$. This is because linear independence of the basis tells us that $n \leq k$, so anything that is polynomial in $n$ is also polynomial in $k$.
The proofs will follow this general outline:
\begin{center}
\begin{itemize}
    \item[] For integer bases:
    \begin{itemize}[$\bullet$]
        \item Using the determinant to define the size of a basis
        \item Using the size of a basis to show the number of iterations is polynomially bounded
        \item Showing the time for each iteration is polynomially bounded
    \end{itemize}
    \item[] For rational bases:
    \begin{itemize}[$\bullet$]
        \item Showing the algorithm is correct for rational bases by using integer bases
        \item Showing the algorithm is still polynomial time even with the extra bits from numerators and denominators
    \end{itemize}
\end{itemize}
\end{center}
We want to show the algorithm works as desired for rational inputs because this is most practical for any implementation of the LLL algorithm --- computers can only store real numbers to some finite number of decimal places, so effectively they are always stored as rational numbers with a small degree of error. Also, oftentimes real numbers can't even be used to construct a lattice: for example, integer linear combinations of 1 and $\sqrt{2}$ are dense in $\R$, so they do not form a set of discrete points.

We are making the standard assumption that each element of the vectors in the input basis can be expressed in a constant number of bits, so bit size is not an important consideration for the runtime of the algorithm on integer bases. However, bit size will become important when proving the results for rational bases, since it is not immediately obvious why the operations in the algorithm do not result in an exponential increase in bit size for the numerators or denominators in the basis as the algorithm progresses. If the numbers were exponentially increasing in bit size, then standard operations on those numbers would take exponential time, forcing the algorithm to no longer be polynomially bounded.

We now begin the argument in the case where the input to the algorithm is an integer basis. The method of proof used here will involve assigning a positive integer ``size'' to the basis, and showing that the size is reduced by at least a fixed factor with every iteration. Since the size is a positive integer, it cannot go below 1, so we can use the fixed reduction factor to compute an upper bound of the number of iterations of the algorithm. A sensible size function can be expressed in terms of determinants, using vectors from the input basis.

\begin{defn} By a slight abuse of notation, we can associate a basis $B$ with a matrix $B$, whose columns are the vectors of the basis in their given order. Then, we define the \textit{determinant} of a basis $B$
\[
D(B) = \sqrt{\det(B^T B)},
\]
where $\det(\cdot)$ is the standard matrix determinant.
\end{defn}
This definition works even when the corresponding matrix $B$ is not square, but we must consider the non-square case separately when proving the following property.

\begin{lem} For any basis $B$ that can be represented as a square matrix,
\[
D(B) = \prod_{i=1}^n||\vec{b}^*_i||.
\]
\end{lem}
\begin{proof}First, we can show that $D(B) = |\det(B)|$. This follows from the fact that $\det(B^T) = \det(B)$, so
\[
    D(B) = \sqrt{\det(B^TB)} = \sqrt{\det(B)^2} = |\det(B)|.
\]

Then, recall from the Gram-Schmidt orthogonalization process that
\begin{align*}
\vec{b}^*_1 &= \vec{b}_1\\
\vec{b}^*_2 &= \vec{b}_2 - \mu_{2, 1}\vec{b}^*_1\\
\vec{b}^*_3 &= \vec{b}_3 - \mu_{3, 2}\vec{b}^*_2 - \mu_{3, 1}\vec{b}^*_1\\
&\vdots\\
\implies \vec{b}_i &= \sum_{j=1}^{i} \mu_{i, j}\vec{b}^*_j,
\end{align*}
noting that $\mu_{i, i} = 1$. Since each of the original basis vectors can be expressed as a linear combination of the orthogonalized basis vectors, we can verify by observation that the following matrix decomposition holds:

\[
    B = \left[\begin{array}{@{}c|c|c@{}}
        \vec{b}_1 & \dots  & \vec{b}_n
        \end{array}\right] =
        \left[\begin{array}{@{}c|c|c@{}}
        \vec{b}^*_1 & \dots  & \vec{b}^*_n
        \end{array}\right]
        \begin{bmatrix} 
        \mu_{1, 1} & \mu_{2, 1} & \dots  & \mu_{n, 1}\\
        0 & \mu_{2, 2} & \dots  & \mu_{n, 2}\\
        \vdots & \vdots & \ddots & \vdots\\
        0 & 0 & \dots & \mu_{n, n} 
        \end{bmatrix}
\]
\[
=
        \left[\begin{array}{@{}c|c|c@{}}
        \frac{\vec{b}^*_1}{||\vec{b}^*_1||} & \dots  & \frac{\vec{b}^*_n}{||\vec{b}^*_n||}
         \end{array}\right]
        \begin{bmatrix} 
        ||\vec{b}^*_1|| & 0 & \dots  & 0\\
        0 & ||\vec{b}^*_2|| & \dots  & 0\\
        \vdots & \vdots & \ddots & \vdots\\
        0 & 0 & \dots & ||\vec{b}^*_n||
        \end{bmatrix}
        \begin{bmatrix} 
        1 & \mu_{2, 1} & \dots  & \mu_{n, 1}\\
        0 & 1 & \dots  & \mu_{n, 2}\\
        \vdots & \vdots & \ddots & \vdots\\
        0 & 0 & \dots & 1 
        \end{bmatrix}
\]

Each vector in the leftmost matrix is orthonormal by this construction, so the determinant of the matrix is $\pm1$. Since the rightmost matrix is upper triangular, its determinant is the product of its entries along the main diagonal, which is also 1. Likewise, the determinant of the middle matrix is the product of entries along the main diagonal. Since the product of determinants gives us the determinant of the product, we can conclude that
\[
\det(B) = \pm\prod_{i=1}^n||\vec{b}^*_i||
\]
\[
\implies |\det(B)| = \prod_{i=1}^n||\vec{b}^*_i||.
\]
\end{proof}

\begin{prop} Lemma 4.2 holds for all bases, even if the rank of the basis is less than the dimension of each vector in the basis.
\end{prop}

\begin{proof}We have already proven Lemma 4.2 for bases corresponding to square matrices, so assume that the rank of the basis, $n$, is less than the dimension of each vector, $k$. In this case, we can augment the matrix $B$ and apply the same process. Take the span of $B$ and find its orthogonal complement, which is a subspace of $\R^k$ of dimension $k-n$. Since every subspace of $\R^k$ has an orthonormal basis, we can choose $k-n$ orthonormal vectors $\vec{v}_1, \ldots,\vec{v}_{k-n}$ that form a basis of this orthogonal complement, defining
\[
\tilde{B} = 
        \left[\begin{array}{@{}c|c|c|c|c|c@{}}
        \vec{b}_1 & \dots & \vec{b}_n & \vec{v}_1 & \dots & \vec{v}_{k-n}
        \end{array}\right].
\]
By this construction, all the $\vec{v}_i$ will be orthogonal to $\vec{b}_1, \ldots, \vec{b}_n$. 
Since Gram-Schmidt orthogonalization on the first $n$ vectors of $\tilde{B}$ does not change their span, the $\vec{v}_i$ will also be orthogonal to $\vec{b}_1^*, \ldots, \vec{b}_n^*$, meaning their projection coefficients onto previously orthogonalized vectors will always be 0. So, all ${v}_i$ will be unchanged under Gram-Schmidt orthogonalization, and thus
\[
\tilde{B}^* = \left[\begin{array}{@{}c|c|c|c|c|c@{}}
        \vec{b}_1^* & \dots & \vec{b}_n^* & \vec{v}_1 & \dots & \vec{v}_{k-n}
        \end{array}\right].
\]
Since $\tilde{B}$ is a square matrix, it satisfies the condition of Lemma 4.2, so its determinant is equal to the product of the lengths of the Gram-Schmidt orthogonalized vectors. Since all of the $\vec{v}_i$ have length 1, we get that
\[
D(\tilde{B}) = \prod_{i=1}^n||\vec{b}_i^*||.
\]

Now, consider the product $\tilde{B}^T\tilde{B}$. When carrying out the multiplication, we can take advantage of the fact that the dot product of any of the $v_i$ with another vector other than itself is equal to zero due to orthogonality, and the dot product with itself is 1. So, we get that
\[
\tilde{B}^T\tilde{B} = \begin{bmatrix} 
        \vec{b}_1\cdot\vec{b}_1 & \dots & \vec{b}_1\cdot\vec{b}_n &  &  &\\
        \vdots & \ddots & \vdots &  &  & \\
        \vec{b}_n\cdot\vec{b}_1& \dots & \vec{b}_n\cdot\vec{b}_n &  &  & \\
         &  &  & 1 &  &  \\
         &  &  &  & \ddots & \\
         &  &  &  &  & 1 \\
        \end{bmatrix}
\]
We can reduce this determinant, for example using expansion by minors, to get that
\[
\det(\tilde{B}^T\tilde{B}) = \det\left(\begin{bmatrix} 
        \vec{b}_1\cdot\vec{b}_1 & \dots & \vec{b}_1\cdot\vec{b}_n\\
        \vdots & \ddots & \vdots\\
        \vec{b}_n\cdot\vec{b}_1& \dots & \vec{b}_n\cdot\vec{b}_n\\
        \end{bmatrix}\right) = \det(B^TB)
\]
So, we can conclude that
\[
D(B) = \sqrt{\det(B^TB)} = \sqrt{\det(\tilde{B}^T\tilde{B})} = D(\tilde{B}) = \prod_{i=1}^n||\vec{b}_i^*||.
\]
\end{proof}

Using this definition of the determinant, we can now construct the size function
\[
S(B) = \prod_{i=1}^n D(\{\vec{b}_1, \vec{b}_2, \ldots, \vec{b}_i\})^2 = \prod_{i=1}^n\prod_{j=1}^i ||\vec{b}_j^*||^2.
\] 

\newgeometry{bottom=1in, left=1.5in, right=1in}
\begin{lem}For the integer basis $B$ at any step of the LLL algorithm, $S(B)$ is always a positive integer.
\end{lem}

\begin{proof}We can associate the subset $\{\vec{b}_1, \vec{b}_2, \ldots, \vec{b}_i\}$ of the basis $B$ with the matrix $B_i$. Then, since $B$ is an integer basis, $B$ and $B_i$ have only integer entries, so $B_i^TB_i$ has only integer entries and thus has integer determinant. So, $D(\{\vec{b}_1, \vec{b}_2, \ldots, \vec{b}_i\})^2 = \det(B_i^TB_i)$ is an integer for all $i$.

Since $S(B)$ is the product of integers, it itself must be an integer. After each iteration of the algorithm, the basis may change, but the vectors remain within the integer lattice defined by the initial integer basis. Therefore, the vectors will always have integer components, which makes $S(B)$ an integer. In particular, it is a positive integer, since each factor is squared.
\end{proof}

\begin{lem}If we take an integer basis $B$ and the basis $B'$ produced after the reduction step in a single iteration of the LLL algorithm to $B$, then
\[
    S(B') = S(B).
\]
\end{lem}

\begin{proof}By Lemma 3.1, we know that the Gram-Schmidt orthogonalization process results in the same vectors when applied to either $B$ or $B'$. Since $S(B)$ and $S(B')$ can be defined purely in terms of the lengths of the Gram-Schmidt orthogonalized vectors, the claim follows.
\end{proof}

\begin{lem} If we take an integer basis $B$ and the basis $B'$ produced after performing one swap in $B$ during the LLL algorithm, then
\[
    S(B') < \delta S(B),
\]
where $\delta$ is the constant in the ordering condition of a $\delta$-LLL reduced basis.
\end{lem}

\begin{proof}
Observe that when we swap two vectors $\vec{b}_i, \vec{b}_{i+1} \in B$ to create $B'$, the size $S(\{\vec{b}_1, \vec{b}_2, \ldots, \vec{b}_j\})$ is unchanged whenever $j \not= i$. This is because when $j < i$, none of the vectors in the basis are changed, and when $j > i$, the basis has the two vectors $\vec{b}_i, \vec{b}_{i+1}$ swapped, but the exact same vectors are still in the basis, so the determinant may become negative, but the squared determinant is unchanged. Therefore, the only difference between $S(B)$ and $S(B')$ comes from the sublattice with basis $\{\vec{b}_1, \vec{b}_2, \ldots, \vec{b}_i\}$. That is,
\[
\frac{S(B')}{S(B)} = \frac{\det(\{\vec{b}_1, \ldots, \vec{b}_{i-1}, \vec{b}_{i+1}\})^2}{\det(\{\vec{b}_1, \ldots, \vec{b}_{i-1}, \vec{b}_{i}\})^2} = \frac{||\vec{b}_{i+1}-\sum_{j=1}^{i-1}\mu_{i+1, j}\vec{b}_j^*||^2}{||\vec{b}_i^*||^2} = \frac{||\vec{b}_{i+1}^* + \mu_{i+1, i}\vec{b}_i^*||^2}{||\vec{b}_i^*||^2}.
\]
Since $\vec{b}_i^*$ is orthogonal to $\vec{b}_{i+1}^*$, by the Pythagorean theorem we get that
\[
\frac{||\vec{b}_{i+1}^* + \mu_{i+1, i}\vec{b}_i^*||^2}{||\vec{b}_i^*||^2} = \frac{||\vec{b}_{i+1}^*||^2 + ||\mu_{i+1, i}\vec{b}_i^*||^2}{||\vec{b}_i^*||^2}.
\]
\restoregeometry
We only swap two vectors in the LLL algorithm when the ordering condition is \textit{not} satisfied; therefore, we can use the fact that $(\delta - \mu_{i+1, i}^2)\Vert\vec{b}_i^*\Vert^2 > \Vert\vec{b}_{i+1}^*\Vert^2$ (using the indices assigned before the swap) to give us the inequality
\[
\frac{||\vec{b}_{i+1}^*||^2 + ||\mu_{i+1, i}\vec{b}_i^*||^2}{||\vec{b}_i^*||^2} < (\delta - \mu_{i+1, i}^2) + \mu_{i+1, i}^2 = \delta
\]
\[
\implies S(B') < \delta S(B).
\]
\end{proof}

This lemma induces the upper bound of $\delta < 1$ introduced earlier in the definition of a $\delta$-LLL reduced basis --- if we want to guarantee that the size will decrease with every swap, we can't have $\delta > 1$. If $\delta = 1$, then we only know that $S(B') < S(B)$ after every swap, so the algorithm could possibly run infinitely if $S(B') - S(B)$ gets arbitrarily small. Thus, we choose $\delta < 1$ to prevent this issue. We also choose $\delta > 1/4$ so that $\delta - \mu_{i+1, i}^2$ is strictly positive, because allowing it to become negative or zero guarantees the ordering condition to be satisfied in those situations, rendering it pointless.

We can now synthesize the lemmas to bound the number of iterations of the algorithm. Lemma 4.5 tells us that we do not need to worry about the size increasing or decreasing in the reduction step of the algorithm, so we can focus only on the swaps. By repeated application of Lemma 4.6, after $s$ swaps, the resulting basis $B^{(s)}$ will satisfy $S(B^{(s)}) < \delta^s S(B) $. Since we proved that $S(B^{(s)})$ is a positive integer for any $s$ in Lemma 4.4, it must always be greater than or equal to 1. Therefore, the number of swaps $s$ occurring over the course of the algorithm must always be small enough such that
\begin{align*}
1 &\leq S(B^{(s)}) < \delta^s S(B)\\
\implies s &< \log_{1/\delta}{S(B)}.
\end{align*}

This shows that not only is the algorithm guaranteed to terminate, but there are $O(\log S(B))$ swaps that will occur over the course of the algorithm. In an iteration where no swap is conducted, the index $i$ is incremented by 1 (refer to Section 3). So, the algorithm will not go more than $n$ consecutive iterations between swaps, or else $i$ would be greater than $n$ and the algorithm would terminate.

Therefore, since each swap occurs at most every $n$ iterations, the total number of iterations of the algorithm is $O(n \log S(B))$. The size $S(B),$ as defined by the product of Gram-Schmidt orthogonalized basis vectors, can be upper bounded using the following lemma:

\begin{lem}The basis vector $\vec{b}_i$ is at least as long as the Gram-Schmidt orthogonalized vector $\vec{b}_i^*.$
\end{lem}

\begin{proof}This follows from the definition of a basis vector in terms of the orthogonalized basis vectors:
\[
\vec{b}_i  = \sum_{j=1}^{i}\mu_{i, j}\vec{b}_j^* = \vec{b}_i^* + \sum_{j=1}^{i-1}\mu_{i, j}\vec{b}_j^*.
\]
$\vec{b}_i$ has the same component as $\vec{b}_i^*$ in the direction of $\vec{b}_i^*$, but has (potentially zero-length) orthogonal components in other directions as well. Therefore, $||\vec{b}_i|| \geq ||\vec{b}_i^*||$.
\end{proof}

So, we get that
\begin{align*}
S(B) &= \prod_{i=1}^n \prod_{j=1}^i ||\vec{b}_j^*||^2\\
&\leq \max_{j\in [1..n]}||\vec{b}_j^*||^{2\cdot\frac{(n+1)n}{2}}\\
&\leq \max_{j\in [1..n]}||\vec{b}_j||^{n^2+n}.
\end{align*}

Next, we can bound the maximum length of an input basis vector. We assumed each entry can be represented in $c$ bits, and the bit length of the product of two numbers cannot exceed the sum of the bit length of each of the factors. In addition, when adding two numbers together, the result can only have at most one more bit than the largest addend. So, the norm squared can have bit length at most $2c + k - 1$, and therefore the norm itself could not possibly have size more than $2^{2c+k-1}$. We can conclude that the number of iterations of the LLL algorithm on an integer basis is
\[
O\left(n \log \max_{j\in [1..n]}||\vec{b}_j||^{n^2+n}\right) = O\left((n^3 + n^2)\log (2^{2c+k-1})\right) = O(k^4),
\]
which is polynomial in $k$  (note that this bound is not perfectly tight, but for this paper any polynomial bound will suffice).

Now, it remains to show that each iteration itself runs in polynomial time. First, we can observe that the Gram-Schmidt orthogonalization process only uses a polynomial number of calculations. This is because each orthogonalized vector $\vec{b}_i^*$ is computed using $\vec{b}_i$ and a linear combination of vectors $\vec{b}_j^*$ for $j < i$, each of which numbers no more than $n$. The coefficients in the linear combination are computed with operations on all $k$ entries of each of the vectors, and once an orthogonalized vector is computed, it does not change. So, the entire orthogonalized basis $B^*$ can be computed in $O(kn^2) = O(k^3)$ time, and we only need to do this once every iteration. The reduction step of the LLL algorithm takes a polynomial number of calculations since the projection coefficients $\mu_{i, j}$ for all $j < i$ are computed, of which there are no more than $n$, taking $O(k)$ time each. Finally, the swapping step requires an inequality check of vector lengths and potentially swapping two vectors, which are both polynomial time.

We have proven that for an integer basis, the LLL algorithm terminates in a polynomial number of iterations for $\delta \in (1/4, 1)$, and that each iteration takes a polynomial amount of time. Therefore, the entire algorithm is polynomially bounded for integer bases.

 Now, we can extend this fact from the integer case to the rational case. The important observation to make is as follows:
 
 For any rational basis $B = \{\vec{b}_1, \vec{b}_2, \ldots, \vec{b}_n\}$ with each $\vec{b}_i = \left(\frac{c_{i, 1}}{d_{i, 1}}, \frac{c_{i, 2}}{d_{i, 2}}, \cdots, \frac{c_{i, k}}{d_{i, k}}\right)$, $c_{i, j}, d_{i, j} \in \Z$, we can produce an analogous integer basis by multiplying each basis vector by the least common multiple of all the denominators. That is, we can define $d = \text{lcm}_{i \leq n, j \leq k}d_{i, j}$ and create a new basis $dB = \{d\vec{b}_1, d\vec{b}_2, \ldots, d\vec{b}_n\}$, which is an integer basis. We want to show that the LLL algorithm does the same operations on $B$ as it would do on $dB$, for which we have already proved the polynomial time bound.

\begin{lem} If the Gram-Schmidt orthogonalized basis of $B$ is equal to $\{\vec{b}_1^*, \vec{b}_2^*, \ldots, \vec{b}_n^*\}$, then the Gram-Schmidt orthogonalized basis of $dB$ is equal to $\{d\vec{b}_1^*, d\vec{b}_2^*, \ldots, d\vec{b}_n^*\}$.
\end{lem}

\begin{proof}We can prove this by induction on the number of vectors. By definition, the first vector in the orthogonalized basis is $d\vec{b}_1 = d\vec{b}_1^*.$ Then, assuming the claim holds for the first $i$ orthogonalized vectors, the formula for the $(i+1)$th orthogonalized vector is given by
\[
(d\vec{b}_{i+1})^* = d\vec{b}_{i+1} - \sum_{j=1}^{i} \mu_{i+1, j}d\vec{b}^*_j.
\]
Observe that the projection coefficients $\mu_{i,j}$ are unchanged by this operation of multiplying by a scalar, since it affects both the numerator and the denominator and cancels out. Therefore,
\[
(d\vec{b}_{i+1})^* = d\left(\vec{b}_{i+1} - \sum_{j=1}^{i} \mu_{i+1, j}\vec{b}^*_j\right) = d\vec{b}_{i+1}^*.
\]
So, the claim holds for the entire basis by strong induction.
\end{proof}

Just as it has been proven for Gram-Schmidt orthogonalization, it can be easily verified that the remaining steps in the LLL algorithm are the exact same after scalar multiplication; that is, the same multiples of the same (scaled) vectors are being added to each other and the same vectors are being swapped as if there was no scalar multiplication. Furthermore, the two criteria for a $\delta$-LLL reduced basis still hold through scalar multiplication of the entire basis. So, the termination criteria being satisfied for $B$ is equivalent to the termination criteria being satisfied for $dB$. Therefore, if the algorithm terminates for $dB$, it must also terminate for $B$ in the same number of steps. Note that $S(dB) = d^{n^2+n}S(B)$, and that $d$ is a least common multiple of all the denominators, so its size is at most the product of all the denominators, which has bit length at most $cnk$ and thus size at most $2^{cnk}$. 
So, the calculation of the number of iterations changes to
\[O(n\log S(dB)) = O(n\log d^{n^2+n} + n\log S(B))
\]
\[
= O(n(n^2+n)\log 2^{cnk} + k^4) = O(k^5),
\]
which is still polynomial in $k$. So, it follows directly that the LLL algorithm terminates in a polynomial number of iterations for all rational bases. Also, note that Proposition 3.2 holds regardless of the type of numbers in the input basis, so the output will in fact be a $\delta$-LLL reduced basis in the rational case.

Now, it remains to argue that the extra number of bits involved in operations on rational numbers with both a numerator and a denominator do not affect the polynomial bound on the number of bit operations of the LLL algorithm when applied to a rational basis. We can also look to the basis $dB$ for this claim. Since $d$ has bit length at most $cnk$, when multiplying each vector by $d$, we increase the bit length of each numerator of each entry in each vector to $cnk + c$ at most, and the denominators are eliminated. Then, each standard operation takes $O(cnk+c) = O(k^2)$ bit operations instead of a constant number, but since we have polynomially many iterations, the total number of bit operations would then be $O(k^2)$ multiplied by some polynomial bounding the iterations, so the runtime of the LLL algorithm on $dB$ is still polynomial in $k$. Looking back to the original rational basis $B$, at any step of the algorithm we can express $B$ as $dB/d$; in this representation, each rational number in each vector has at most $cnk + c$ bits in the numerator, and at most $cnk$ bits in the denominator, for a total of $2cnk + c$ bits, which is still $O(k^2)$. Therefore, even if we input a rational basis, the number of bit operations that will be used in the algorithm is still polynomially bounded. So, we have proven the entire result:

\begin{thm}For all rational input bases, the LLL algorithm is guaranteed to terminate and output a $\delta$-LLL reduced basis in polynomial time for $\delta \in (1/4, 1)$. \qed
\end{thm}

\section{Application of LLL algorithm to SVP}

Now that we have discussed the algorithm and proven that it works as desired, we can reap the benefits of the reduced basis. In particular, we can produce a reasonable lower bound for the length of the shortest vector in any lattice. We must briefly prove a preliminary lemma.\medskip

\begin{lem}
    For any vector $\vec{v} \in \Lambda(B)$, there exists $\vec{b^*_i} \in B^*$ such that $ ||\vec{b^*_i}|| \leq ||\vec{v}||.$
\end{lem}

\begin{proof}We can express $\vec{v}$ as a linear combination of basis vectors, and then express each of those basis vectors as a linear combination of Gram-Schmidt orthogonalized vectors to get the desired conclusion. That is,
\begin{align*}
\vec{v} &= \sum_{i=1}^n z_i \vec{b}_i,\,\,\,\,z_i \in \Z \\
&= \sum_{i=1}^n \sum_{j=1}^i z_i \mu_{i, j}\vec{b}^*_j
\end{align*}
Let $m$ be the greatest value of $i$ for which $z_i$ is non-zero. Then, the above is equal to
\begin{align*}
    &\sum_{i=1}^m \sum_{j=1}^{i} z_i\mu_{i, j}\vec{b}_j^*.
\end{align*}
In this sum, $z_m\vec{b}_m^*$ is the only term with a multiple of $\vec{b}_m^*$, so all the other vectors in the sum are orthogonal to it. This means that adding on some multiple of those vectors strictly increases the length of the resultant vector. Therefore, since $z_m$ is a non-zero integer, it has absolute value greater than or equal to 1, and we can conclude
\[
||\vec{v}|| \geq ||z_m\vec{b}^*_m|| \geq ||\vec{b}^*_m||,
\]
so $\vec{b}^*_m$ satisfies the claim.
\end{proof}

\begin{thm}If we have a lattice $\Lambda(B)$ with $\delta$-LLL reduced basis $\{\vec{b}_1, \vec{b}_2, \ldots, \vec{b}_n\}$, then 
\[
     \left(\delta - \frac{1}{4}\right)^{\frac{n-1}{2}}\Vert\vec{b}_{1}\Vert \leq \Vert\lambda(B)\Vert,
\]

where $\lambda(B)$ denotes the shortest vector in $\Lambda(B)$.
\end{thm}
\begin{proof}We have that $(\delta - \mu_{i+1, i}^2)\Vert\vec{b}_i^*\Vert^2 \leq \Vert\vec{b}_{i+1}^*\Vert^2$ from the reduced basis condition (2), which implies that
\[
\Vert\vec{b}_i^*\Vert^2 \leq \frac{\Vert\vec{b}_{i+1}^*\Vert^2}{\delta - \mu_{i+1, i}^2} \leq \frac{\Vert\vec{b}_{i+1}^*\Vert^2}{\delta - \frac{1}{4}}
\]
By repeated application of that inequality, we get that
\[
\Vert\vec{b}_1\Vert^2 = \Vert\vec{b}_1^*\Vert^2 \leq \frac{\Vert\vec{b}_{i}^*\Vert^2}{(\delta - \frac{1}{4})^{i-1}} \leq \frac{\Vert\vec{b}_{i}^*\Vert^2}{(\delta - \frac{1}{4})^{n-1}}
\]
for all $i \leq n$. So, we can see that
\[
\Vert\vec{b}_{1}\Vert^2 \leq \frac{\min_{i\leq n}\Vert\vec{b}_{i}^*\Vert^2}{(\delta - \frac{1}{4})^{n-1}}\]
\[\implies \left(\delta - \frac{1}{4}\right)^{\frac{n-1}{2}}\Vert\vec{b}_{1}\Vert \leq \min_{i\leq n} \Vert\vec{b}_{i}^*\Vert,\]
and by Lemma 5.1, since it applies to any vector in $\Lambda(B)$, and certainly $\lambda(B) \in \Lambda(B)$, we conclude
\[
\left(\delta - \frac{1}{4}\right)^{\frac{n-1}{2}}\Vert\vec{b}_{1}\Vert \leq \Vert\lambda(B)\Vert.
\]
\end{proof}

So, the LLL algorithm allows us to succinctly bound the length of the shortest vector in any lattice with rational basis by using a multiple of the first vector in the $\delta$-LLL reduced basis. While the coefficient is exponential in $n$, which isn't necessarily that tight of a bound, the ability to produce the reduced basis in polynomial time makes it one of the best available options for many applications.

\section{Further applications of the LLL algorithm}

The LLL algorithm can be applied to a wide array of problems, and for many of these problems, the common theme is that we are aiming to find some object that is equal to or close to zero. We can apply LLL in these cases by constructing a lattice where each vector in the lattice represents one of the objects that we want to consider, such that shortness of a vector is equivalent to the object being close to zero. With such a construction, the problem is then reduced to finding the shortest vector in the lattice. A simple example will help make this idea more concrete.

\textbf{6.1. Approximating minimal polynomials.} An algebraic number (over $\Z$) is a number that is a root of a polynomial with integer coefficients, like $3, \frac{1}{2}, \sqrt{2},$ or the golden ratio $\varphi$ (satisfying $x^2 - x - 1 = 0$). The minimal polynomial of an algebraic number $\alpha$ is the monic polynomial with integer coefficients of lowest degree that has $\alpha$ as a root (e.g. $x^2 - x - 1$ is the minimal polynomial of $\varphi$). The LLL algorithm allows us to quickly find the minimal polynomial for some algebraic number $\alpha$ if we are given a number close to $\alpha$. For example, take $\alpha = \sqrt{2} \approx 1.414$. $\sqrt{2}$ has minimal polynomial $x^2 - 2$, so we hope to produce this polynomial with the LLL algorithm. We can construct a lattice with basis
\[
\left\{\begin{pmatrix}1\\0\\0\\1000\cdot1.414^2
\end{pmatrix}, \begin{pmatrix}0\\1\\0\\1000\cdot1.414
\end{pmatrix}, \begin{pmatrix}0\\0\\1\\1000
\end{pmatrix}
\right\}.
\]
The reason we choose this basis is because a general vector in its lattice has the form
\[
a\begin{pmatrix}1\\0\\0\\1000\cdot1.414^2
\end{pmatrix} + b\begin{pmatrix}0\\1\\0\\1000\cdot1.414
\end{pmatrix} + c\begin{pmatrix}0\\0\\1\\1000
\end{pmatrix}
\]
\[
= \begin{pmatrix}a\\b\\c\\1000(a\cdot1.414^2 + b\cdot1.414 + c)
\end{pmatrix}
\]
for $a, b, c \in \Z$. Each vector then corresponds to the polynomial $ax^2 + bx + c$, where the first three elements in the vector represent the coefficients, and the fourth element represents the polynomial evaluated at our approximation of 1.414, scaled up by a factor of 1000. The only way for such a vector to have a small length is when $1000(a\cdot1.414^2 + b\cdot1.414 + c)$ is close to zero, so $a\cdot1.414^2 + b\cdot1.414 + c \approx 0$, implying that 1.414 is close to a root of the polynomial $ax^2 + bx + c$. Therefore, finding a short vector in the lattice is equivalent to finding a polynomial of degree at most 2 such that a number close to 1.414 is a root of that polynomial, and furthermore, that the coefficients of that polynomial are relatively small. That describes the minimal polynomial of $\sqrt{2}$, or at least some polynomial with small coefficients in which $\sqrt{2}$ evaluates to a very small number.

Using LLL on this basis with $\delta = 3/4$, we get the reduced basis
\[
\left\{\begin{pmatrix}-1\\0\\2\\0.604
\end{pmatrix}, \begin{pmatrix}10\\-12\\-3\\25.96
\end{pmatrix}, \begin{pmatrix}-15\\29\\-11\\15.06
\end{pmatrix}
\right\}.
\]
Sure enough, the first three elements of the first vector in this basis exactly give us the (negative) coefficients of the minimal polynomial $x^2 - 2$. The fourth element of the first vector is so small because it is equal to $1000(-1.414^2 + 2) \approx 1000(-(\sqrt{2})^2 + 2) = 0$. Note that we need the factor of 1000 for this method to work, because it forces the algorithm to ``prioritize'' minimizing the value of the polynomial in the fourth element before the coefficients are minimized, since there is no way a vector in the lattice can be small without the polynomial in the fourth element being very close to zero. Here, 1000 is somewhat arbitrary, but if a very small factor is used, it is possible that a polynomial with very small coefficients and a root close to $\sqrt{2}$ --- but not equal to it --- could be produced, which we want to avoid.

In general, this process works for any algebraic number $\alpha$ if we are given an approximation $\alpha' \approx \alpha,$ for any degree of minimal polynomial, although we must have a guess for the degree beforehand to set the size of the basis. If we want to search for a minimal polynomial of $\alpha$ of degree $d-1$, then letting $\vec{e_1}, \vec{e_2}, \ldots, \vec{e_d}$ be the standard basis vectors in $\R^d$, we can append an extra element to each of the standard basis vectors to construct the lattice basis
\[
\left\{\begin{pmatrix}\vec{e_1}\\10^k(\alpha')^{d-1}
\end{pmatrix}, \begin{pmatrix}\vec{e_2}\\10^k(\alpha')^{d-2}
\end{pmatrix},\cdots, \begin{pmatrix}\vec{e_d}\\10^k
\end{pmatrix}
\right\}
\]
for some suitably large $k$. Then, after applying LLL to this basis, the first vector in the reduced basis will correspond to the minimal polynomial of $\alpha$, or something close to it.

You can also apply this method to gather evidence for whether any number is transcendental (i.e. not algebraic). One can construct a lattice in exactly the same way, substituting the number to be tested as $\alpha,$ and increasing the size of the basis to increase the range of possible degrees of polynomials spanned. If the LLL algorithm does not output a reduced basis with a vector that has a 0 as its last element, then none of the vectors represent a polynomial for which $\alpha$ is exactly a root. This doesn't constitute a proof, because it doesn't \textit{guarantee} that $\alpha$ does not satisfy any integer polynomial of the degree specified, but it shows that it is highly unlikely, with the probability increasing as the factor of 10 in the fourth element is increased. 

\textbf{6.2. Further reading.} While most are outside the scope of this paper, there are many other applications of LLL. Some involve constructing complex lattices and using LLL on those lattices like in the example above, while others take advantage of the hardness of the the SVP, as well as the bound on the SVP that LLL induces. These applications include:
\begin{itemize}[topsep=0pt]
    \item Factoring polynomials with rational coefficients, which, in fact, was the title of the original paper by Lenstra, Lenstra, and Lovász that introduced the LLL algorithm [3].
    \item Using continued fractions to approximate real numbers with rational numbers [6].
    \item Finding small roots of integer polynomials mod $N$ (Coppersmith's method) [7][8].
    \item Partially inverting RSA encryption (uses Coppersmith's method) [7][8].
    \item Showing difficulty of solving the learning with errors (LWE) and ring learning with errors (RLWE) problems, even for quantum computers. This problem is used as the foundation for many modern encryption schemes, particularly homomorphic encryption, because it is presumed to be ``quantum-safe'' [9][10]. 
    \item Disproof of Merten's conjecture, which proposes that \[\sum_{k=1}^n\sum\zeta_k < \sqrt{n}\] for all $n > 1$, where the second sum is over all primitive $k$th roots of unity (i.e. $(\zeta_k)^k = 1, (\zeta_k)^m \not=1$ for $m < k$) [11].
    
\end{itemize}

    Several of these applications use the same idea of trying to find a small or zero object, constructing a lattice where shortness of a vector corresponds to smallness of the object, and then obtaining a short vector in that lattice, which must necessarily result in the found object being close to, or equal to, zero. This method was used in 6.1, where we wanted to find a polynomial, both with small coefficients and a small output at a particular $\alpha,$ so we constructed a lattice where both outcomes would be equivalent to finding a short vector. The LLL algorithm gives us a polynomial time method to search for an approximate solution to each of these problems, so long as they can be reframed in terms of short vectors in lattices.

\pagebreak

\end{document}